\documentclass[12pt, letterpaper,reqno]{amsart}
\usepackage{amsmath}
\usepackage{amsfonts}
\usepackage{amssymb}
\usepackage{amsthm}

\usepackage{graphicx}
\usepackage{xcolor}
\usepackage{subfig}
\usepackage[margin=1in]{geometry}
\usepackage{enumerate}
\usepackage{mathrsfs}
\usepackage{color}
\newtheorem{theorem}{Theorem}
\theoremstyle{plain}

\newtheorem{corollary}{Corollary}

\newtheorem{lemma}{Lemma}

\numberwithin{equation}{section}

\usepackage{fancyhdr}
\newcommand{\bcd}{\[\begin{tikzcd}}
\newcommand{\ecd}{\end{tikzcd}\]}

\newcommand{\Z}{{\mathbb{Z}}}
\newcommand{\xnk}{X_n(k)}
\newcommand{\xrk}{X_r^k}
\newcommand{\prk}{P_r^k}

\newcommand{\xsk}{X_s^k}
\begin{document}
\title[Spectrum of Rooted Homogeneous Trees]{On the Spectrum of Finite, Rooted Homogeneous Trees}
\author{Daryl DeFord} 
\address{CSAIL\\
Massachusetts Institute of Technology}
\email{ddeford@mit.edu}

\author{Daniel Rockmore} 
\address{Departments of Mathematics and Computer Science\\
Dartmouth College}
\email{rockmore@math.dartmouth.edu}

\date{\today}

\begin{abstract}
In this paper we study the adjacency spectrum of  families of finite rooted trees with regular branching properties. In particular, we show that in the case of constant branching, the eigenvalues are realized as the roots of a family of generalized Fibonacci polynomials and produce a limiting distribution for the eigenvalues as the tree depth goes to infinity. We indicate how these results can be extended to periodic branching patterns and also provide a generalization to higher order simplicial complexes. 

\end{abstract}

\maketitle

\section{Introduction} For  integers $r,k>1$,  let $\xrk$ denote the {\em $k$-ary rooted tree of depth $r$}. That is, $\xrk$ is first of all a tree (in the graph-theoretic sense -- so a connected graph without cycles) with a distinguished node, the ``root" that orients all the nodes in $\xrk$ in the sense that any node is located at some {\em depth} defined as the distance of that node from the root.  As usual, the {\em leaves} of $\xrk$ are the nodes most distant from the root and have degree $1$. The {\em tree is of depth $r$}, if the leaves are at distance $r$ from the root. The rooted tree $\xrk$ is $k$-ary in the sense that every non-leaf node has $k$ {\em children} defined as the nodes connected to it that are at depth one more than the node itself. Thus, an interior node is also a {\em parent} of $k$ children. A node at distance $d$ from the root is of {\em generation $d$}. Note that in $\xrk$ the root has degree $k$, interior nodes have degree $k+1$ and leaves have degree $1$. The well-known {\em binary tree} of depth $r$ is $X_r^2$. 

This kind of regular branching -- constant number of children at every non-leaf node -- has natural extensions. For example, one can consider more general periodic branchings: given a vector $\bar{k} = (k_0,k_1,\dots, k_{q-1})$, define the rooted tree $X_r^{\bar{k}}$ to be the tree that has successions of branchings with $k_0$ children of the root, $k_1$ children for the nodes at level $1$ and in general,  a node of generation $aq+b$ where $0 \leq a \leq r$  and $0\leq b < q$  has $k_b$ children. 

In this paper we consider properties of the spectrum of the trees $X_r^{\bar{k}}$, focusing mainly on the simple case of $\bar{k} = (k_0)$, but where possible, extending to vectors of longer length and finite subtrees of well--known infinite trees. By ``spectrum" we mean the spectrum of the adjacency operator of the tree (graph). These trees can be viewed as simplicial complexes of dimension $1$ with some structural regularities. This view admits at least one natural two-dimensional generalization (with corresponding adjacency operator) and we also include some preliminary investigations of this higher order context in Section 5.2.

\subsection{Related Work}

In the literature, the trees studied here are frequently referred to as {\em Bethe Trees} and there has been an active body of research centered around characterizing the spectral properties of operators defined on these trees 
[1, 12-19].   Many of the techniques in these papers make use of algebraic factorizations of the characteristic polynomial. These factorizations then reduce to a collection of tridiagonal eigenvalue computations that can be evaluated as a determinant using a product formula. In this paper we take a combinatorial approach, giving explicit constructions of the eigenvalues and eigenvectors. This allows us to determine the spectral properties of these operators in terms of roots of recurrent families of polynomials and prove limiting results about the corresponding spectral measures. In some cases, our results mirror those from the algebraic perspective, for example Theorem 7 in \cite{R2} is equivalent to Lemma 1 combined with Theorem 3 in this paper and Theorems 9 and 10 of \cite{R7} are similar to our derivations in Section 4.  

Homogeneous trees  with constant branching have found application in a wide variety of fields. In particular, finite versions of these trees where all non-leaf vertices have the same number of edges have found frequent application in computer science. The infinite case is also well studied. The vector-valued form has been studied extensively in the context of dynamical systems -- wherein their consideration arises  naturally in the context of the dynamics of infinitely renormalizable maps on the unit interval as well as periodic orbit structures for maps on the unit interval \cite{BORT}. The $(p+1)$-regular infinite tree is a homogeneous space for the $p$-adic Lie group $SL_2(\Z_p)$ wherein the associated Hecke operator takes the place of the adjacency matrix and is well-studied for various number theoretic connections \cite{cartier}.

Our results are also related to many recent results on the expansion properties of graphs. See \cite{HLW} for a survey and examples of how methods from trees can be extended to arbitrary regular graphs. These methods are mostly concerned with deriving bounds on the spectral gap for regular graphs rather than computing the entire spectrum explicitly. Additionally, there has been recent work on similar problems for higher order simplices \cite{PR}.

\begin{figure}[!h]
\subfloat[$X_4^2$]{\includegraphics[height=2in]{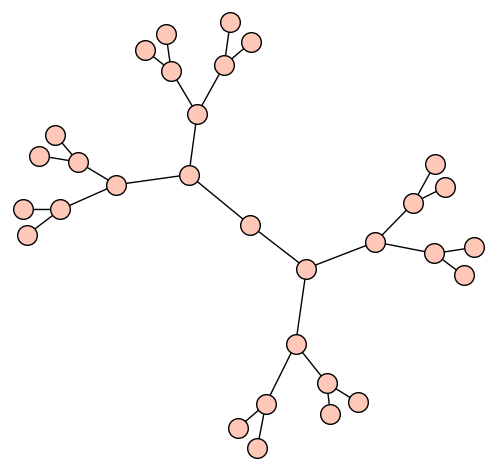}}\qquad\qquad
\subfloat[$\widehat{X_2^4}$]{\includegraphics[height=2in]{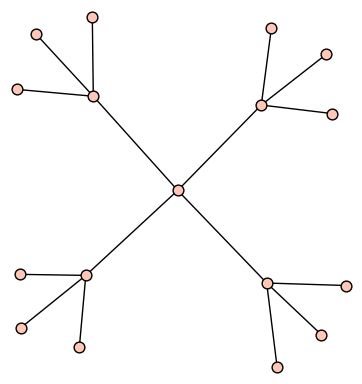}}
\caption{The two principal types of rooted homogeneous trees considered in this paper. Note that $X_r^k$ each non--leaf has $k$ children while in the subtrees of the infinite regular tree $\widehat{X_r^k}$ each non--leaf has degree $k$.}
\label{fig:simple_trees}
\end{figure}

\subsection{Main Results}

Our primary results are for the rooted trees $\xrk$ For $r<s$ the rooted tree $\xrk$ naturally embeds into any rooted tree $\xsk$ by identifying the roots. We will see that the spectra of these trees also have a natural nesting and moreover, use the nested structure to relate the associated eigenvectors as well and  characterize the eigenvalues as roots of families of polynomials derived from the recursive structure of the trees.

To be more precise, for $k\geq 2$, define polynomials $P^{k}_n(x)$  via the recurrence 
\begin{equation}
P^{k}_n(x)= x P^{k}_{n-1}(x)-kP^{k}_{n-2}(x) \label{eq:pndef1}
\end{equation}
with initial conditions 
\begin{equation}
\left\{\begin{array}{l}
P^{k}_0(x) =0,\\
P^{k}_1(x) =1. \\
\end{array}\right.\label{eq:pndef2}
\end{equation}

Thus, for $n>0$, $P^{k}_n(x)$ is of degree $n-1$.

\begin{theorem} Let $r \geq 1$ and $k\geq 2$. The the roots of $P^{k}_2(x),\dots,P^{k}_{r+2}(x)$   are precisely the  eigenvalues of the  $k$-ary tree $\xrk$. \end{theorem}

\begin{corollary} If $r \leq r'$ then the spectrum of $\xrk$ is a subset of the spectrum of $X_{r'}^k$. 
\end{corollary}

\begin{theorem} If $\lambda$ is a root $P^{k}_r$ and not a root of $P^{k}_m$ for any $m<r$ then asymptotically (as $r\longrightarrow \infty$), the proportion of eigenvalues of $\xrk$ equal to $\lambda$ is $\dfrac{(k-1)^2}{k^r-1}$.
\end{theorem}

Together, these facts allow us to show that the sequence of spectral measures convergence to a singular distribution with well-defined endpoints.\\
{\bf Theorem 4. }
\textit{
The endpoints of the support of $\lim_{n\rightarrow\infty} \mathcal{U}(NSpec(A(X_n^k)))$ are: }

$$\left\{\varepsilon \dfrac{(k-1)^2}{k^m-1} + \sum_{n=1}^\infty \dfrac{(k-1)^2}{k^n-1}\sum_{\substack{(\ell,n)=1\\ \ell<\frac{an}{m}}} 1 : \begin{array}{c} m\in \mathbb{N} \\ (a,m)=1 \\\varepsilon\in\{0,1\}\end{array} \right\}.$$

\section{Spectrum of $\xrk$}

\subsection{Preliminaries.} 

Let $L^2(X_r^k)$ denote the $|\xrk |$-dimensional vector space of complex-valued functions on $\xrk$.  Counting the nodes by depth gives that
\begin{equation}
|\xrk | = \dfrac{{k^{r+1} - 1}}{{k-1}}.
\end{equation}

The {\em standard basis} for $L^2(\xrk)$ is the set of functions $e_\gamma$ indexed by the nodes $\{  \gamma\}$ of $\xrk$ that are equal to $1$ at the given node $\gamma$ and $0$ elsewhere. Thus any $v\in L^2(\xrk)$ can be written
\begin{equation}
v = \sum_{\gamma \in \xrk} v_\gamma e_\gamma
\label{eq:stdbasis}
\end{equation}
so that $v$ is thus identified with the values that it takes on the nodes of $\xrk$. This is the natural representation of $v$ for considering the effect of the adjacency operator
\begin{equation}
(Av)_\gamma = \sum_{\beta \sim \gamma} v_\beta
\end{equation}
where $\beta \sim \gamma$ denotes that $\beta$ is adjacent (linked) to $\gamma$. 

We call $v\in L^2(\xrk)$ {\em  isotropic} (with respect to the root) if $v_\gamma$ is constant over all $\gamma$  at a given distance from the root. In that case, we let $v_m$ denote the common value that $v$ takes for $\gamma$ at distance $m$ from the root. It is easy to see that $A$ takes isotropic functions to isotropic functions. 

The spectrum of $A$ will be characterized by the zeros of the polynomials $P_n^k(x)$. These are exactly the generalized Fibonacci polynomials of type $U(x,-(k-1))$ defined in \cite{HL}.  These polynomials were introduced in \cite{WP} and many of their general properties were reported in  \cite{FHM, HL}. In Lemma \ref{lem:fibpoly} we record some of the properties that we will use throughout the paper.

\begin{lemma}
\label{lem:fibpoly}
The polynomials $P_n^k(x)$ have the following properties:
\begin{enumerate}
\item If $\lambda$ is a root of $P_m^k(x)$ and $(m,n)=1$ then $\lambda$ is not a root of $P_n^k(x)$.
\item The polynomials $P_n^k(x)$ form a divisibility sequence: for $m|n$, then $P_m^k(x) | P_n^k(x)$.
\item The roots of $P_m^k(x)$ are exactly $\left\{2\sqrt{k}\cos(\dfrac{\ell\pi}{m}): 1\leq \ell\leq m\right\}$. 
\end{enumerate}
\end{lemma}
\begin{proof} These results are Theorems 4, 6, and 10 in \cite{HL}.
\end{proof}

\subsection{Eigenvalues for $\xrk$ and roots of $\prk(x)$.} We first show the ``forward" direction of Theorem 1:  if $\lambda$ is a root of $P^{k}_s(x)$  for $1\leq s\leq r$ then $\lambda$ is an eigenvalue of $\xrk$. We do this by explicitly constructing eigenvectors for these $\lambda$ using the standard basis.

\begin{lemma}
Let $v\in L^2(\xrk)$ be isotropic with $v_m = P_{r-(m-1)}^k(\lambda)$ for $\lambda$ a root of $P^{k}_{r+2}(x)$. Then $v$ is an eigenvector for $\xrk$ with eigenvalue $\lambda$.
\end{lemma}

\begin{proof}
Let $v\in L^2(\xrk)$ be defined as in the statement of the lemma. We now confirm the eigenrelation $(Av)_m = \lambda v_m$ for $0\leq m \leq r$. We will make repeated use of the defining properties of the polynomials (see Eqns. (\ref{eq:pndef1}), (\ref{eq:pndef2})). 

\noindent {\bf Case 1:} $m=r$. In this case we are at the leaves. Thus 
$$(Av)_r = v_{r-1} = P_{r-(r-1-1)}^k(\lambda) = P_{2}^k(\lambda) = \lambda P_1^k(\lambda) = \lambda v_r.$$

\medskip

\noindent {\bf Case 2:} $0 < m < r$. In this case we are at a node that is neither a leaf nor the root.  Then
$$
\begin{array}{rcl}
(Av)_m & = & v_{m-1} + (k-1)v_{m+1} \\
& = & P_{r-(m-1-1)}^k(\lambda) + k P_{r-(m+1-1)}^k(\lambda)\\
& = & \lambda P_{r-m+1}^k(\lambda)  =\lambda P_{r-(m-1)}^k(\lambda) \\
& = & \lambda v_m
\end{array}$$
where the second-to-last line uses the defining recurrence of the $\prk(x)$. 

\medskip
\noindent {\bf Case 3:} $m=0$. In this case, we are at the root and 
$$\begin{array}{rcl}
(Av)_0 &=& k v_1\\
& = & kP_{r-(1-1)}^k(\lambda) = kP_{r}^k(\lambda)\\
& = & \lambda P_{r+1}^k(\lambda) - P_{r+2}^k(\lambda)\\
& = & \lambda P_{r-(0-1)}^k(\lambda) = \lambda v_0
\end{array}$$
where the second-to-last line uses the definition of the  $\prk(x)$ and the fact that $\lambda$ is a root of 
$P_{r+2}^k(x)$.

\end{proof}

Note that the only case of the proof that depended on the choice of $\lambda$ as a root of $P_{r+2}^k(x)$ was Case 3. We now show how this eigenvector construction can be extended  to the roots of the other $P_\ell^k(x)$. In short, for any distance $d\geq 0$ from the root, the construction can be modified by constructing vectors that are linear combinations of analogous functions, isotropic with respect to rooted {\em subtrees} issuing from a node at distance $d$. 
 We'll do it in pieces. 

\begin{lemma}
If $\lambda$ is a root of $P^{k}_{r+1}(x)$ then $\lambda$ is an eigenvalue of $\xrk$. Moreover, the dimension of the $\lambda$-eigenspace  is $k-1$. 
\end{lemma}

\begin{proof}
As in the previous lemma we construct eigenvectors $v$ directly. We begin by selecting $k$ real numbers $\{a_1,a_2,\ldots,a_k\}\subset\mathbb{C}$  such that $\sum_{i=1}^k a_i=0$ and not all the $\{a_i\}$ are zero.  Note that this is a $k-1$-dimensional subspace.  Order the children of the root arbitrarily from $1$ to $k$ and for each descendant of child $i$ at distance $m$ from the root, set the corresponding eigenvector value to $v_i= a_i P^{k}_{r-(m-1)}(\lambda)$. Finally, set the value at the root to $v_{0} = P_{r+1}^k(\lambda)=0$. 

In this case, the eigenrelation at the root is satisfied since the sum of the values assigned to the children of the root is $\sum_{i=1}^k a_i P^{(k)}_{r}(\lambda)=P^{(k)}_{r}(\lambda)\sum_{i=1}^k a_i=0$.  Note that while $v$ is not isotropic on $\xrk$, it is isotropic on any  {\em subtree rooted at any  child of the root} and extending through all of its descendants. The argument in the previous lemma is easily adapted:  the only change is that each of the relations for the various cases are multiplied by $a_i$ except possibly for the case of the root's children themselves -- the new ``root" of the subtree. For this, consider any given child of the root $i$: 

$$\begin{array}{rcl}
(Av)_{i} &= &v_0 + a_i\cdot k \cdot P^k_{r-(2-1)}(\lambda)\\
&=& 0 + a_i(\lambda P^k_{r}(\lambda) - P^k_{r+1}(\lambda))\\
&=& a_i\lambda P^k_{r}(\lambda) = \lambda v_i
\end{array}$$
where the second line follows from (\ref{eq:pndef1}) and the third line follows from the choice of $\lambda$ and the definition of $v$.

\end{proof}

The construction of Lemma 3 can be generalized: instead of only prescribing the root to be zero, we can now construct eigenfunctions that are zero out to some fixed distance from the root and then mimic the construction. We say an eigenvector $v$, with associated eigenvalue $\lambda$, of $\xrk$ is {\em of type $(\lambda, d)$} if in its standard basis representation it has at least one nonzero coefficient for a node (index) at distance $d-2$ from the leaves and no nonzero entries at distance greater than $d-3$ from the leaves. As an example, the eigenvectors constructed in Lemma 2 are of type $(\lambda, r+1)$ for $\lambda$ a root of $P_{r+1}^{(k)}$.

\begin{lemma}
If $\lambda$ is a root of $P^{k}_s(x)$ for $s<r+1$ then there is a $k^{r-s}(k-1)$-dimensional eigenspace for $\lambda$ consisting of eigenvectors with type $(\lambda, s)$ for $\xrk$. 
\end{lemma}
\begin{proof}

To construct these eigenvectors we begin by setting the coefficient at the root and the coefficients corresponding to   all nodes whose distance to the root is less than or equal to $r-s+1$ to zero. Recall that there are precisely  $k^{r-s}$ nodes at distance exactly $r-s+1$ from the root.  For each node $j$  at distance $r-s+1$ from the root, we select $k$ complex numbers $\{a^j_1, a^j_2, \ldots,a^j_{k}\}$ so that $\sum_{\ell=1}^{k}a^j_\ell=0$. 
  
Following the construction given in Lemma 2 we order the children of each of these nodes arbitrarily and for each descendant of node $j$ at distance $m$ we set the corresponding eigenvector coefficient to $a^{j}_i P^{(k)}_{s-m}(\lambda)$. By the arguments given in Lemma 1 and Lemma 2 these vectors satisfy the eigenvector equation for the adjacency operator. Since we have $k-1$ degrees of freedom at each of the $k^{r-s}$ nodes this proof is complete.

\end{proof}

\begin{theorem} The roots of $P^{k}_s(x)$  for $1\leq s\leq r+2$ 
are precisely the  eigenvalues of the finite $k$-ary tree. \end{theorem}

\begin{proof}
It suffices to show that we have constructed $|\xrk | = {\frac{k^{r+1} - 1}{k-1}}$ linearly independent eigenvectors corresponding to these values. Summing up over the eigenvectors of type $(\lambda, s)$ from Lemma 4, recalling that $P^{k}_s(x)$ is a degree $s-1$ polynomial gives $\sum_{s=1}^{r} k^{r-s}(k-1)(s) $ eigenvalues. This sum telescopes, leaving $k^r+k^{r-1}\cdots+k-r$ eigenvalues accounted for. Adding the remaining $r+1$ eigenvalues from $P^{k}_{r+2}(x)$ completes the proof.

\end{proof}

\subsection{Limiting Distribution} In this section, we explore the behavior of the distribution of the eigenvalues as $n\rightarrow\infty$. As we saw above, the multiplicities of the eigenvalues that correspond to roots of $P_n^{(k)}$ continue to grow exponentially with $n$. Qualitatively, a  (suitably normalized) plot for large $n$ of the eigenvalues of $A(X^k_n)$ with multiplicities should look like a collection of horizontal line segments. We begin by determining how the number of unique values grows as $n\rightarrow \infty$.

\begin{lemma}
The number of roots of $P^k_n$ that are not roots of $\prod_{i=1}^{n-1} P^k_{i}$ is equal to $\varphi(n)$.
\end{lemma}
\begin{proof}
Note that $P^k_n$ is a polynomial of degree $n-1$.   
Since $P^k_n$ is a divisibility sequence, if $n$ is prime then $P^k_n$ is irreducible and has $n-1=\varphi(n)$ roots. We proceed by strong induction on $n$. Since $2$ and $3$ are prime and have roots $\{0\}$ and $\{\pm\sqrt{k-1}\}$ respectively our base case is complete.  For the inductive step consider the number of new roots of $P^k_n$. There are $n-1$ total roots and each divisor of $n$ contributes $\varphi(n)$ by the inductive hypothesis: 
\begin{align*}
\#\ \textrm{of new roots}=& \; n-1-\sum_{d|n : d\notin\{1,n\}}\varphi(n)\\
= & \; n-1+\varphi(n)+\varphi(1)-\sum_{d|n}\varphi(n)\\
=& \; n-1+\varphi(n)+1-n\\
=& \; \varphi(n)
\end{align*}
as desired.

\end{proof}

Next, we determine the number of occurrences of each of the new values from the previous lemma. The reason that this question is not answered by Lemma 4 is that since $P^k_n$ is a divisibility sequence, the values continue to reoccur. For determining the total number of eigenvalues we did not have to consider the effects of these repetitions, as the corresponding eigenvectors as in Theorem 4 are linearly independent. However, now we are interested in counting the total multiplicities of the values themselves.

\begin{lemma}  Let $\lambda$ be a root of $P^k_m$ that occurs for the first time and let $\delta_{a,m}(n)$ be the function that returns 1 if $n\equiv a\pmod{m}$ and 0 otherwise. Then, for $n>m$ the eigenvalue $\lambda$ occurs with multiplicity 

$$\sum_{i=1}^{\lfloor \frac{n+1}{m}\rfloor} (k-1)k^{n-mi+1}=\dfrac{(k-1)k^{m+n+1}}{k^m-1}\left( 1-k^{-m\lfloor\frac{n+1}{m} \rfloor}\right)+\delta_{1,m}(n).$$

 Equivalently, for large $n$  the proportion of eigenvalues of $X^k_n$ of value $\lambda$ converges to $\dfrac{(k-1)^2}{k^{m}-1}$.

\end{lemma}
\begin{proof} 
Let $\lambda$ be a root of $P^k_m$ that occurs for the first time. By the divisibility property we know that $\lambda$ is a root of $P^k_{m\ell}$ for all $\ell\in \mathbb{N}$. From Lemma 4 this gives that the total multiplicity of $\lambda$ as an eigenvalue for $X^k_n$ is  $\sum_{i=0}^{\lfloor \frac{n+1}{m} \rfloor} (k-1)k^{n-mi+1} +\delta_{1,m}(n)$.

Rewriting this as a geometric series allows us to simplify:
\begin{align*}
\sum_{i=1}^{\lfloor \frac{n+1}{m} \rfloor} (k-1)k^{n-mi+1} +\delta_{1,m}(n)=&\delta_{1,m}(n)+(k-1)k^{n+1}\sum_{i=1}^{\lfloor \frac{n+1}{m} \rfloor} k^{-mi}\\
=& \delta_{1,m}(n)+(k-1)k^{n+1}\left( \dfrac{1-k^{-m\lfloor\frac{n+1}{m} \rfloor}}{{1-k^{-m}}}\right)\\
=& \delta_{1,m}(n)+\dfrac{(k-1)k^{m+n+1}}{k^m-1}\left( 1-k^{-m\lfloor\frac{n+1}{m} \rfloor}\right).\\
\end{align*}
Since the total number of nodes in $X^k_n$ is ${\frac{k^{r+1} - 1}{k-1}}$ this gives the asymptotic proportion as $\dfrac{(k-1)^2}{k^{m}-1}$.
\end{proof}

These results have an interesting corollary,  a version  of which has been previously proved using Lambert Series:

\begin{corollary}
$$\sum_{n=2}^\infty \dfrac{\varphi(n)(k-1)^2}{(k)^n-1}=1.$$
\end{corollary}
When $k=2$ (the binary tree) this gives the following simple sum:
\begin{corollary}
$$\sum_{n=1}^\infty \dfrac{\varphi(n)}{2^n-1}=2.$$
\end{corollary}

The previous propositions  explain the  characteristic ``Devil's Staircase'' shape  of ordered plot of eigenvalues of $X^{k}_n$. That is, for each of the $\varphi(m)$ values $\lambda$ that are roots of $P^{(k)}_m$ occurring for the first time, the width of the ``stair'' corresponding to each $\lambda$ in the plot is proportional to $\dfrac{(k-1)^2}{k^{m}-1}$.

\subsection{Singular Distributions}
 Another way to interpret the results of this section is to consider the sequence of uniform distributions over the eigenvalues, with multiplicities, of $A(X^k_n)$ as $n\rightarrow\infty$. The condition in Definition 1 gives that the limit of these distributions is a   singular distribution, like the Cantor function.
 In order to formalize this discussion it is convenient to normalize the eigenvalues of $A(X^k_n)$ and consider the set:
 $$NSpec(A(X_n^k))=\left\{\dfrac{\lambda+k}{2k} : \lambda\in Spec(A(X_n^k)) \right\}.$$

Then the sequence of uniform distributions over $NSpec(A(X_n^k))$ converges to a singular distribution as $n\rightarrow\infty$. This raises the interesting question of the properties of the measure zero support set for the limiting distribution. Lemmas 6 and 7 allow us to write expressions for these endpoints. 

\begin{theorem}
The endpoints of the support of $\lim_{n\rightarrow\infty} \mathcal{U}(NSpec(A(X_n^k)))$ are: 

$$\left\{\varepsilon \dfrac{(k-1)^2}{k^m-1} + \sum_{n=1}^\infty \dfrac{(k-1)^2}{k^n-1}\sum_{\substack{(\ell,n)=1\\ \ell<\frac{an}{m}}} 1 : \begin{array}{c} m\in \mathbb{N} \\ (a,m)=1 \\\varepsilon\in\{0,1\}\end{array} \right\}.$$
\end{theorem}

\begin{proof}
We proceed by directly computing the endpoints of each interval corresponding to a new root of $P^{(k)}_n$. Lemma 6 gives that the width of the corresponding interval is $\dfrac{(k-2)^2}{(k-1)^{m}-1}$, so it suffices to compute the left endpoint. By part (iii) of Lemma 1 we know that these values are of the form $1\sqrt{k}\cos(\frac{a \pi}{m})$ where $(a,m)=1$. Since cosine is monotonic on $[0,\pi]$ and $\frac{\ell \pi}{m}\in [0,\pi]$ for all $\ell<m$ the intervals that occur to the left of any given value $2\sqrt{k}\cos(\frac{a \pi}{m})$ are exactly those for which $\cos(\frac{\ell \pi}{n} ) < \cos(\frac{a \pi}{m}) \rightarrow \frac{\ell}{n}<\frac{a}{m} $.

Thus, summing the widths over all possible values of $n$ the left endpoint of the interval is given by: $$\sum_{n=1}^\infty \dfrac{(k-1)^2}{k^n-1}\sum_{\substack{(\ell,n)=1\\\ell<\frac{an}{m}}} 1$$
and the right endpoints follow by adding the width of the relevant interval. 
\end{proof}

\begin{figure}[!h]
\subfloat[$A(X_4^5)$]{\includegraphics[height=2in]{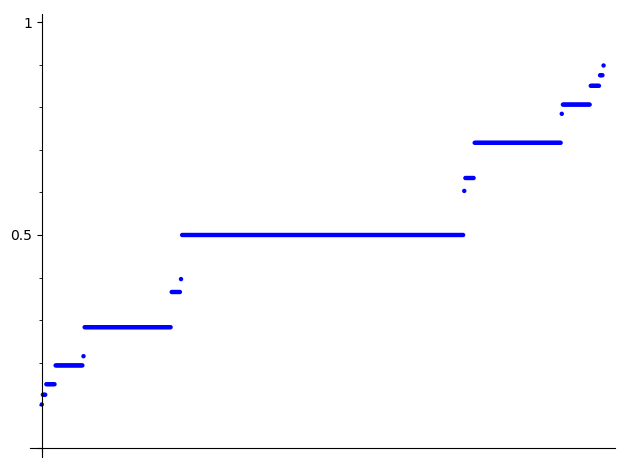}}\qquad\qquad
\subfloat[$A(\widehat{X_2^6})$]{\includegraphics[height=2in]{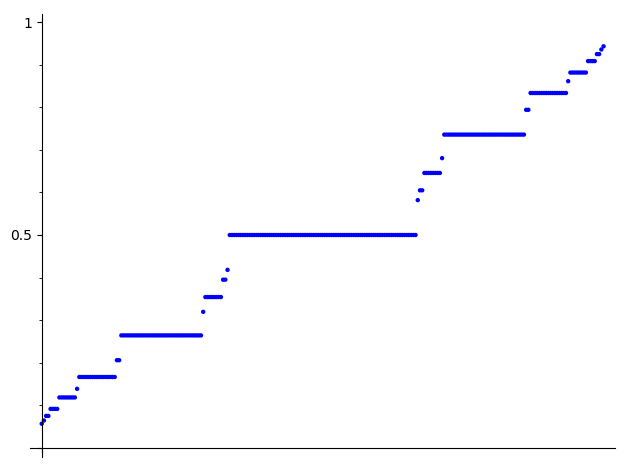}}
\caption{Normalized eigenvalues of adjacency spectra. Notice the transient  eigenvalues associated to the $Q$ polynomials in plot (B) for the $k$-regular (except for the leaves) version $\widehat{X^k_r}$. }
\label{fig:tree_eigs}
\end{figure}

\section{Finite Subtrees of Infinite Regular Trees}

Another natural set of trees to consider is the collection of regular rooted finite subtrees of the infinite regular trees. The difference between these and those considered above is that the trees described in this case the root has $k$ children but each other non--leaf node in these trees has $k-1$ children so that each non--leaf has exactly $k$ neighbors. These are the regular finite subtrees that are encountered in the context of buildings and $p$-adic homogeneous spaces \cite{buildings}. We will denote these trees by $\widehat{X^{k}_r}$.

To represent the eigenvalues of these trees in terms of polynomials as in Section 2 we need to define another family of polynomials in order to satisfy the eigenvector relation at the root.  
For $n >0$, let $Q^k_n(x)$ be a polynomial of degree $n$ satisfying the recurrence  \begin{equation}
 Q^{k}_n(x)=xP^{k}_n(x)-kP^{k}_{n-1}(x) 
 \end{equation}\label{eq:qndef}
 with initial conditions $Q^{k}_1(x)=x$ and $Q^{k}_2(x)=x^2-k$. 

The main results of the previous setting can be translated to these trees with minimal adjustments to the proofs.
\begin{lemma}
Let $v\in L^2(\widehat{\xrk})$ be depth regular with $v_m = P_{r-(m-1)}^k(\lambda)$ for $\lambda$ a root of $Q^{k}_{r+1}(x)$. Then $v$ is an eigenvector for $\widehat{\xrk}$ with eigenvalue $\lambda$.\end{lemma}

\begin{lemma}
If $\lambda$ is a root of $P^{k}_{r+1}(x)$ then $\lambda$ is an eigenvalue of $\widehat{\xrk}$. Moreover, the dimension of the $\lambda$-eigenspace 
is $k-1$.
\end{lemma}

\begin{lemma}
If $\lambda$ is a root of $P^{k}_s(x)$ for $s<r+1$ then there is a $(k-2)k(k-1)^{r-s}$ dimensional eigenspace for $\lambda$ consisting of eigenvectors with type $(\lambda, s)$ for $\widehat{\xrk}$. 
\end{lemma}

\begin{lemma}
The dimension of the eigenspaces of $\widehat{X^{k}_r}$ corresponding to eigenvalues that are roots of $P^{(k)}_s$ for $s\leq n$ are equal to:
$$\dfrac{k}{k-2}\left((k-1)^n-1\right)-n.$$
\end{lemma}

\begin{proof}

Let $m_n$ be the sum of the multiplicities of the eigenvalues of $X^{k}_n$ that correspond to roots of $P^{k}_s$ for $s\leq n$. We claim that $m_n$ satisfies the recurrence relation $m_{n}=(k-1)m_{n-1}+(k-2)n+1$. We can divide the eigenvalues into three classes: those that occur with multiplicity at least $k$ in $X^{k}_{n-1}$,  those that occur with multiplicity exactly $k-1$ in $\widehat{X^{k}_{n-1}}$, and those that do not occur in $\widehat{X^{k}_{n-1}}$.  

For the eigenvalues in case 1, Lemma 3 tells us that each value occurs with $(k-1)$ times the multiplicity in $\widehat{X^{k}_{n}}$. For the eigenvalues in case 2, Lemma 3 tells us that $(k-1)m_n$ overcounts by a factor of $k(k-2)(n-1)-(k-1)^2(n-1)=n-1$. Finally, for case 3, Lemma 2 gives us exactly $(k-1)n$ new linearly independent eigenvectors. Putting these cases  together we get $m_{n+1}=(k-1)m_n-(n-1)+(k-1)n=(k-1)m_n+(k-2)n+1$. 

This is a recurrence relation with eigenvalues $\{(k-1),1,1\}$ so it can be represented as a generalized power sum  \cite{CombR} of the form $m_na+bn+c(k-1)^n$. Using the initial conditions $m_0=0$, $m_1=k-1$, and $m_2=k^2-2$ we can compute $a=\frac{-k}{k-2}$, $b=-1$, and $c=\dfrac{k}{k-2}$ which gives $m_n=\dfrac{k}{k-2}\left((k-1)^n-1\right)-n$ as desired. 
\end{proof}

\begin{theorem} The roots of $P^{k}_s(x)$  for $1\leq s\leq r$ and $Q^{k}_r(x)$ are precisely the  eigenvalues of the finite  rooted subtrees of the $k$-ary tree. \end{theorem}

\begin{proof}

Lemma 10 provides $\dfrac{k}{k-2}\left((k-1)^n-1\right)-n$ eigenvalues of $\widehat{\xrk}$. Since $Q^{k}_n$ is a degree $n+1$ polynomial, the construction in Lemma 1 gives us the remaining eigenvectors. 
\end{proof}

\subsection{Limiting Distribution} As in the cases discussed in Section 2 the reoccurrence of the eigenvalues in the divisibility sequence means that the limiting distribution of the eigenvalues for these trees can be computed. Interestingly, the limit identities of Corollaries 2 and 3 are also expressed in these trees with the change of variable $k-1$ for $k$. Although this is a little unintuitive due to the addition of the $Q$ eigenvalues, note that the size of the tree grows exponentially in $r$ while the degree of $Q$ grows linearly. Thus, asymptotically, almost all of the eigenvalues of trees of sufficiently large depth are roots of $P_s^k$ for some $s$.

\begin{lemma} Let $\lambda$ be a root of $P^{k}_m$ that occurs for the first time. Then, for $n>m$ the eigenvalue $\lambda$ occurs with multiplicity at least\footnote{It is possible for $\lambda$ to occur as both a root of $P_m^{k}$ and $Q_m^{k}$. For example, there are $0$--eigenvalues of both types in $X^3_2$. } $$\sum_{i=0}^{\lfloor \frac{n}{m} \rfloor} r_{n-mi}=\dfrac{k(k-2)\cdot(k-1)^{n+m}}{(k-1)^{m}-1}\left(1-(k-1)^{-m\lfloor\frac{n}{m}\rfloor}\right)+(k-1)\delta_{0,m}(n))$$
in $Spec(A(\widehat{X^{k}_n})$. Equivalently, for large $n$  the proportion of eigenvalues of $A(\widehat{X^{k}_n})$ of value $\lambda$ is approximately $\dfrac{(k-2)^2}{(k-1)^{m}-1}$.

\end{lemma}
\begin{proof} 
Let $\lambda$ be a root of $P^{k}_m$ that occurs for the first time. By the divisibility property we know that $\lambda$ is a root of $P^{k}_{m\ell}$ for all $\ell\in \mathbb{N}$. From Lemma 4 this gives that the total multiplicity of $\lambda$ as an eigenvalue for $A(\widehat{X^{k}_n})$ is at least $\sum_{i=0}^{\lfloor \frac{n}{m} \rfloor} r_{n-mi}$.

Rewriting this as a geometric series allows us to simplify:
\begin{align*}
\sum_{i=0}^{\lfloor \frac{n}{m} \rfloor} r_{n-mi}=&\sum_{i=0}^{\lfloor \frac{n}{m} \rfloor} k(k-2)(k-1)^{n-mi} +(k-1)\delta_{0,m}(n)\\
=&(k-1)\delta_{0,m}(n)+ k(k-2)\cdot(k-1)^n\sum_{i=0}^{\lfloor \frac{n}{k} \rfloor} (k-1)^{-mi}\\
=&(k-1)\delta_{0,m}(n)+ k(k-2)\cdot(k-1)^n \left( \dfrac{1-(k-1^{-m\lfloor \frac{n}{k} \rfloor}}{1-(k-1)^{-m}}\right)\\
=&\dfrac{k(k-2)\cdot(k-1)^{n+m}}{(k-1)^{m}-1}\left(1-(k-1)^{-m\lfloor\frac{n}{m}\rfloor}\right)+(k-1)\delta_{0,m}(n)).
\end{align*}
Since the total number of nodes in $\widehat{X^{(k)}_n}$ is $\dfrac{k}{k-2}\left((k-1)^n-1\right)+1$ this gives the asymptotic proportion as $\dfrac{(k-2)^2}{(k-1)^{m}-1}$.
\end{proof}

\begin{theorem}
The endpoints of the support of $\lim_{n\rightarrow\infty} \mathcal{U}(NSpec(A(\widehat{X_n^{k}})))$ are: 
$$\left\{\varepsilon \dfrac{(k-2)^2}{(k-1)^m-1} + \sum_{n=1}^\infty \dfrac{(k-2)^2}{(k-1)^n-1}\sum_{\substack{(\ell,n)=1\\ \ell<\frac{an}{m}}} 1 : \begin{array}{c} m\in \mathbb{N} \\ (a,m)=1 \\ \varepsilon\in\{0,1\}\end{array} \right\}.$$
\end{theorem}

\subsection{Eigenvalues on Infinite Trees} The results from earlier in this section can be used to give explicit constructions of eigenvectors for the Hecke operators on the infinite tree that are stable under congruence subgroups of $SL_2(\mathbb{Q}_p)$.

\begin{theorem}
Let $\lambda$ be an eigenvalue of $A(\widehat{X^{k}_n})$ with corresponding eigenvector $v$. Then, $v$ can be extended to a unique eigenvector of $A(X^{k}_\infty)$ stable under $\Gamma(n)$. 
\end{theorem}

\begin{proof}
We construct the associated Hecke eigenvector, $w$, directly by computing the values on the nodes at each distance from the root. We begin by assigning the values from $v$ to $w$ for the corresponding nodes in the infinite tree  $A(X^{k}_\infty)$. As $v$ already satisfies the eigenvalue equation for  $A(\widehat{X^{k}_n})$ we must assign the value $0$ to all nodes at depth $n+1$ in the infinite tree.

 Continuing on, the stability condition enforces that all nodes with the same immediate parent must be assigned to the same value. Thus, for each each leaf of $\widehat{X^{k}_{n}}$ the values assigned to the nodes below the corresponding node at depth $n$ in the infinite tree
 must satisfy $w_{n+i}=\dfrac{\lambda w_{n+(i-1)}-w_{n+(i-2)}}{k-1}$ with initial conditions given by the value of the leaf in $v$ and $0$. This recurrence completely determines $w$ and we have constructed the desired Hecke eigenvector. 

 \end{proof}

\section{Periodic Branching}

We now turn to trees with non--constant branching behavior.
We will see that the spectral behavior of these trees with periodic branching is quite similar to the homogeneous trees discussed earlier. In order to consider sequences of trees that are spectrally stable, we derive the adjacency spectra for trees with complete periods of branching, (e.g., $(3,2,3,2,3,2)$ so that the $3,2$ motif is not interrupted).

We will carefully derive the construction for 2--periodic trees (like one with $(3,2,3,2,3,2)$ branching) and then show how the method can be generalized to longer periods. 

\subsection{$(\alpha,\beta)$ trees}
In this section we construct the eigenvalues and eigenvectors of trees with branching pattern $(\alpha,\beta,\alpha,\beta,\ldots)$ of finite length. As in the case of homogeneous trees we will construct the eigenvalues as roots of a sequence of polynomials. 

We proceed as in Section 2 defining two families of polynomials whose roots are the eigenvalues of the trees. However, the branching process makes determining the relationships between the polynomials at each layer more complex. In order to derive a relation for the polynomials we define a system of recurrences: 

\begin{align*}
A_n=&xB_{n-1}-\alpha A_{n-1}\\
B_n=&xA_n-\beta B_{n-1}
\end{align*}
where $A_n$ represents the nodes with $\beta$ children and the $B_n$ represent nodes with $\alpha$ children. In order to derive a recurrence relation for the $A_n$ and the $B_n$ separately we use the determinant operator method of \cite{CombR} Theorem 7.20. Letting $E$ denote the sequence shift operator, our system can be summarized: 
$$\begin{bmatrix}
E+\alpha&-x\\-Ex&E+\beta
\end{bmatrix}. $$

The determinant of this matrix is $E^2-(x^2-(\alpha+\beta))E+\alpha\beta$ so the sequences individually satisfy the recurrence: $A_n=(x^2-(\alpha+\beta))A_{n-1}- \alpha\beta  A_{n-2}$ and this determines the recurrence that defines our polynomials $P^{(\alpha,\beta)}_n$. Thus, we define 
 $P^{(\alpha,\beta)}_n=(x^2-(\alpha+\beta))P^{(\alpha,\beta)}_{n-2}- \alpha\beta  P^{(\alpha,\beta)}_{n-4}$ with initial conditions $P^{(\alpha,\beta)}_0=0$, $P^{(\alpha,\beta)}_1=1$, $P^{(\alpha,\beta)}_2=x$, and $P^{(\alpha,\beta)}_3=x^2-\beta$ and define $Q^{(\alpha,\beta)}_n=xP^{(\alpha,\beta)}_n-\alpha P^{(\alpha\beta)}_{n-1}$. Note that applied to the tree this becomes a fourth order recurrence that only involves even indexed terms as the $A_n$ and $B_n$ sequences are independent.

 \begin{theorem}
$A(X^{(\alpha,\beta)}_{2n})$  has $\lambda$ as an eigenvalue if and only if $\lambda$ is a root of $P^{(\alpha,\beta)}_m$ for $m\leq 2n$ or a root of $Q^{(\alpha,\beta)}_{2n}$.
\end{theorem}

\begin{proof} 
We follow the outline Lemmas 2--5. As in Section 2, it is straightforward to construct eigenvectors of $X^{(\alpha,\beta)}_{2n}$ that are roots of  $P^{(\alpha,\beta)}_m$ for $m\leq 2n$ or a root of $Q^{(\alpha,\beta)}_{2n}$. The one subtlety to note is that since we are growing the tree two layers at a time there are two sets of roots for the $P^{(\alpha,\beta)}_{2n}$, those for $P^{(\alpha,\beta)}_{2n}$ that occur with non--zero values on the direct children of the root and those of $P^{(\alpha,\beta)}_{2n-1}$ that occur with zero values on the direct children of the root as well. 

In order to show that these are all of the eigenvalues it again suffices to show that we have constructed $|X^{(\alpha,\beta)}_{2n}|$ linearly independent eigenvectors.  Since the layers grow alternately by $\alpha$ and $\beta$ the total number of nodes is: 
$1+\sum_{i=0}^n \alpha^i\beta^{i-1}+\alpha^i\beta^i$.

 Again, since we are growing the tree two layers at a time, we define define $r^{(\alpha,\beta)}_n=(\alpha\beta)^{n-1}(\alpha-1)$ to count the multiplicities of the roots of even indexed polynomial. Similarly, the odd indexed polynomials occur $\alpha(\beta-1)$ times at first appearance and also $(\alpha\beta)$ additional times in each larger tree. Thus, we also define $s^{(\alpha,\beta)}_n=(\alpha\beta)^{n-1}(\beta-1)\alpha$. 
 
Counting the multiplicities with the $2n+1$ roots of $Q^{(\alpha,\beta)}_{2n}$,the $r^{(\alpha,\beta)}_n$, and $s^{(\alpha,\beta)}_n$ gives the following telescoping sum:

\begin{align*}
(2n+1)+\sum_{i=0}^{n-1} (2n-2i)r^{(\alpha\beta)}_{(2n-2*i)} +\sum_{i=0}^{n-1} ((2n-1)-2i)s^{(\alpha\beta)}_{((2n-1)-2i)}=\\
(2n+1)+\sum_{i=0}^{n-1} (2n-2i)(\alpha-1)(\alpha\beta)^i+\sum_{i=0}^{n-1}((2n-1)-2i)(\beta-1)\alpha^{i+1}\beta^i=\\
(2n+1)+\sum_{i=0}^{n-1} (2n-2i)(\alpha^{i+1}\beta^i-(\alpha\beta)^i)+\sum_{i=0}^{n-1}((2n-1)-2i)(\alpha^{i+1}\beta^{i+1}-\alpha^{i+1}\beta^i)=\\
1+\sum_{i=0}^n \alpha^i\beta^{i-1}+\alpha^i\beta^i 
\end{align*}
which completes the theorem. 

\end{proof}

\subsection{Longer Periods}
The analysis for 2--periodic trees can be extended to $\ell$-periodic trees of any period. The matrices in terms of the shift operator are almost circulant matrices with the first $\ell-2$ rows taking the form: 
$$  [\overset{j-1}{\overbrace{0,0,\ldots,0}}, E, -Ex, E\alpha_{j+2}]$$
and final two rows:
$$ [\alpha_1,0,\ldots,0,E, -x]$$
and
$$[-x,\alpha_2,0,0,\ldots,0,E].$$

 In these cases, the recurrence relation of the polynomials associated to the $X^{(\alpha_1,\alpha_2,\ldots,\alpha_{\ell-1},\alpha_\ell)}_{\ell n}$ are given by:
$$E^2-\left(\sum_{i=0}^{\lfloor \frac{\ell}2\rfloor} (-1)^{i}x^{\ell-2i}\sigma_{(\alpha_1,\ldots,\alpha_\ell)}(i+1)  \right)E+\prod_{z=1}^\ell \alpha_z$$
where $\sigma_{(\alpha_1,\ldots,\alpha_\ell)}(i)$ is the sum over all $i$--term products of the $\alpha_z$ with no consecutive indices $\pmod{\ell}$. For example, $$\sigma_{(\alpha_1,\ldots,\alpha_\ell)}(1)=\sum_{j=1}^\ell \alpha_j$$ and $$\sigma_{(\alpha_1,\ldots,\alpha_\ell)}(2)=  \sum_{i=1}^{\ell} \alpha_i\sum_{\substack{i<j\\ |i-j|>1 \!\!\!\pmod{\ell}}} \alpha_j.$$

With these formulas, constructing the polynomials and counting the eigenvalues proceeds as in the proof of Theorem 8.   

\subsection{Recurrence Branching}
Another interesting class of trees are those where the number of branches at each layer grows with $k$. These admit a particularly simple spectral structure in the limit, described in the following theorem. For example, we could consider the ``Fibonacci'' tree which has two branches from the root, followed by three, then five, continuing. In this case, due to the exponentially increasing proportion of leaves, the eigenvalue $0$ comes to dominate the distribution as recorded in the following result.  
\begin{theorem} Let $(\alpha_1,\alpha_2, \alpha_3, \ldots)$ be an increasing sequence.
Then, as $n\rightarrow \infty$ the proportion of eigenvalues of $A(X^{(\alpha_1,\alpha_2, \alpha_3, \ldots)}_n)$ that are equal to zero goes to 1. 
\end{theorem}

\begin{proof}
Setting $\alpha_0=1$, the number of nodes (and equivalently the dimension of the full eigenspace for $A(X^{(\alpha_1,\alpha_2, \alpha_3, \ldots)}_n)$ is equal to $\sum_{i=0}^k \prod_{j=0}^i \alpha_j$. To each set of $\alpha_n$ leaves attached to the same vertex there is a corresponding $(\alpha_n-1)$ dimensional zero--eigenspace giving a lower bound on the total dimension of the zero--eigenspace of $(\alpha_n-1)\prod_{j=0}^{n-1}\alpha_j$.  Thus, it suffices to show that:

$$ \lim_{k\rightarrow \infty} \dfrac{\displaystyle (a_k-1)\prod_{i=0}^{k-1} a_i}{\displaystyle\sum_{i=0}^k\prod_{j=0}^{i} a_j} = 1. $$

We have

$$1>\dfrac{\displaystyle (a_k-1)\prod_{i=0}^{k-1} a_i}{\displaystyle\sum_{{ i}=0}^k\prod_{j=0}^{{ i}} a_j} > \dfrac{\displaystyle (a_k)\prod_{i=0}^{k-1} a_i}{\displaystyle\sum_{{ i}=0}^k\prod_{j=0}^{{ i}} a_j}=\dfrac{\displaystyle\prod_{i=0}^{k} a_i}{\displaystyle\sum_{{ i}=0}^k\prod_{j=0}^{{ i}} a_j}.$$

So it suffices to show that:

$$\lim_{k\rightarrow\infty}\dfrac{\displaystyle\prod_{i=0}^{k} a_i}{\displaystyle\sum_{{ i}=0}^k\prod_{j=0}^{{ i}} a_j}=1$$

or equivalently that:
$$\lim_{k\rightarrow\infty}\dfrac{\displaystyle\sum_{{ i}=0}^k\prod_{j=0}^{{ i}} a_j}{\displaystyle\prod_{i=0}^{k} a_i}=1.$$

Pulling out the leading term, we get: 
\begin{align*}1+\dfrac{\displaystyle\sum_{{ i}=0}^{k-1}\prod_{j=0}^{{ i}} a_j}{\displaystyle\prod_{\ell=0}^{k} a_\ell} =&
1+
\left(\dfrac{1}{a_k}+\dfrac{1}{a_ka_{k-1}}+\cdots+\dfrac{1}{a_ka_{k-1}\cdots a_1} \right)\\
<&1+\dfrac{1}{a_k} +
\left(\dfrac{1}{a_ka_{k-1}}+\dfrac{1}{a_ka_{k-1}}+\cdots+\dfrac{1}{a_ka_{k-1}} \right)\\
=& 1+ \dfrac{k}{a_k}+\dfrac{k-1}{a_k a_{k-1}}.
\end{align*}

Since the $a_k$ are strictly increasing, taking the limit as $k\rightarrow \infty$ gives the result.

\end{proof}

\section{Generalizations and Future Work}

\subsection{Other operators} In addition to the adjacency operator analyzed above, another standard operator on $L^2(X^{k}_n)$ is the {\em graph Laplacian} which  acts on functions on the vertices by $(Lv)_j=\sum_{i\sim j} v(j)-v(i)$. This action can be viewed as a discretization of the continuous Laplacian used for studying heat diffusion. As with the adjacency operator the spectrum of the Laplacian is intrinsically connected to the geometry of a graph.  Additionally, associated to any graph we can form a stochastic matrix that describes the transition probabilities of a simple random walk on the nodes of the graph, where at each time step the walker moves uniformly selects one of the edges adjacent to their current position and travels to the other vertex of the edge.

The machinery developed in Sections 2 and 3 carries forward directly to both of these families of operators. As examples, we state versions of Theorem 5 for these operators on $\xrk$.

\begin{theorem}
Let $\widehat{P^{k}_{n+1}}=(k-x)\widehat{P^{k}_{n}}-(k-1)\widehat{P^{k}_{n-1}}$ with initial conditions $\widehat{P^{k}_{1}}=1$ and $\widehat{P^{k}_{2}}=1-x$ and $\widehat{Q^{k}_n}=(x-k)\widehat{P^{k}_n}+k\widehat{P^{k}_{n-1}} $. Then, $\lambda$ is an eigenvalue for the Laplacian on $\xrk$ if and only if $\lambda$ is a root of  $\widehat{P^{k}_{s}}$ for some $s<k$ or a root of $\widehat{Q^{k}_n}$. 
\end{theorem}

\begin{theorem}
Let $\widetilde{P^{k}_{n+1}}=(k+1)(x\widetilde{P^{k}_{n}}-(k)\widetilde{P^{k}_{n-1}})$ with initial conditions $\widetilde{P^{k}_{1}}=1$ and $\widetilde{P^{k}_{2}}=(k+1)x$ and $\widetilde{Q^{k}_n}=x\widetilde{P^{k}_n}-\widetilde{P^{k}_{n-1} }$. Then, $\lambda$ is an eigenvalue for the random walk matrix on $\xrk$ if and only if $\lambda$ is a root of  $\widetilde{P^{k}_{s}}$ for some $s<k$ or a root of $\widetilde{Q^{k}_n}$. 
\end{theorem}

\subsection{Higher order simplices}
Having  characterized the spectrum of the $k$--ary regular trees we now turn to higher order objects. Both hypergraphs and simplicial complexes can be viewed as generalizations of graphs that encode   more than binary relations between objects and operators, and operators like the adjacency and Laplacian examples considered in the previous section also exist for these objects \cite{FC}. 

Recently there has been significant interest in constructing higher order expanders, generalizing the definition of expansion from the setting of graph to hypergraphs and other generalizations. The survey \cite{AL} summarizes the recent work in this area.

 Viewing a graph as a 1--simplex, we extract the relevant components to construct a 2--simplex version of our problem. To construct the trees   above we began with a single vertex (0--simplex) and then at each stage added new edges (1--simplicies) to each existing leaf vertex to make the resulting graph $k$-regular. Raising the dimension of each component leads us to consider a 2--simplex constructed by beginning with a single edge (1--simplex) and at each stage adding new faces (2--simplices) to make the resulting complex $k$--face regular. This construction can also be viewed as a $k$--regular hypergraph where each hyperedge contains $k$ vertices and each layer is constructed by adding a pair of new vertices/hyperedges for each pair of nodes at the previous step that occur in exactly 1 edge. Finally, this construction can also be viewed as the finite, rooted components of the building for $SL_m(\mathbb{Q}_p)$. We call the simplices derived by this construction {\it the $k$--regular rooted fan of dimension $m$}. 
 
There are two natural notions of adjacency between $k-1$ simplices in a fan, derived from the components of the combinatorial Laplacian \cite{FC}. To extend our earlier results to this setting we consider ``upper'' adjacency where two $(k-1)$-simplices are adjacent if and only if they bound the same $k$-face in the fan. Using this formulation we can construct the graph whose vertices are the $(k-1)$-simplices and where  edges are the defined by this ``upper'' adjacency and determine its spectrum using methods similar to those above. The first three binary fans and their corresponding upper adjacency graphs are shown in Figure \ref{fig:fans}. These upper adjacency graphs are formed by much like the trees above, with copies of the complete graph on $k+1$ nodes replacing the leaves. As there is still a symmetric relationship between the newly added vertices that belong to the same complete graph, we can extend the methods for constructing eigenpairs detailed above to this setting: 

\newpage 

\begin{figure}
    \centering
    \subfloat{\includegraphics[height=2in]{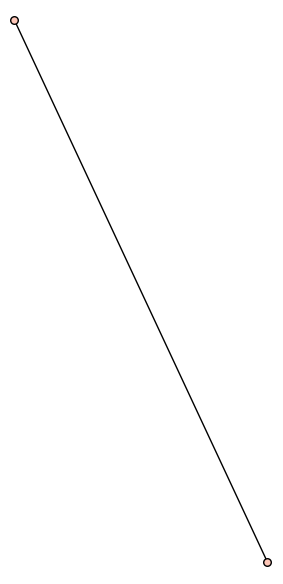}}\qquad\qquad\qquad
        \subfloat{\includegraphics[height=2in]{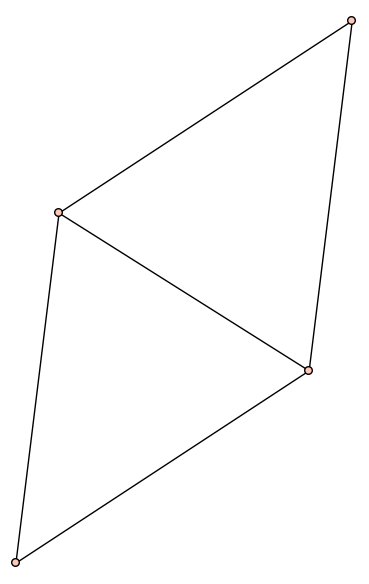}}\qquad\qquad
            \subfloat{\includegraphics[height=2in]{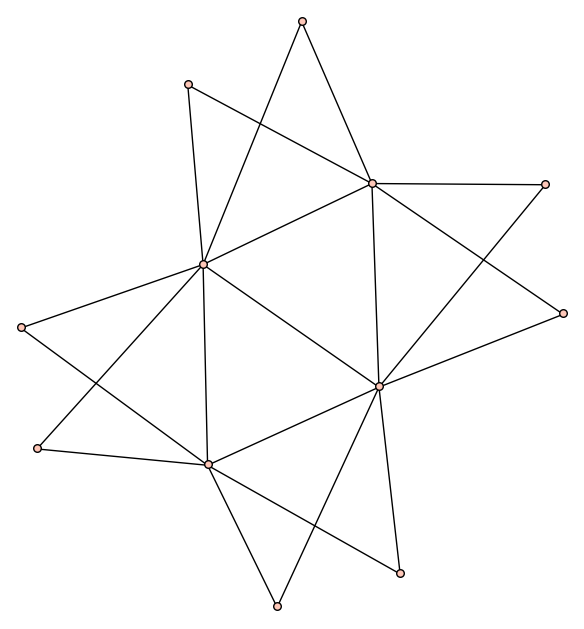}}\\
                \subfloat{\includegraphics[height=2in]{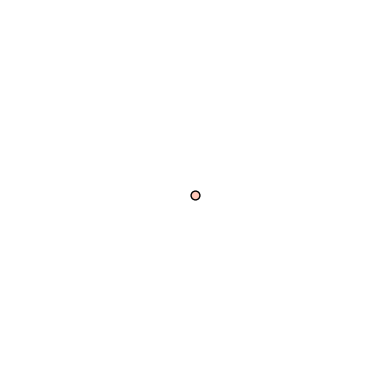}}\qquad
                \subfloat{\includegraphics[height=2in]{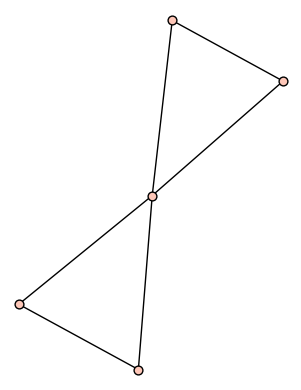}}\quad\quad
    \subfloat{\includegraphics[height=2in]{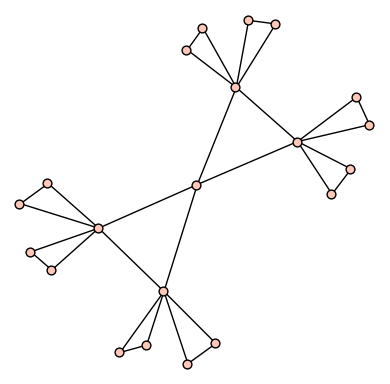}}\quad

    \caption{Rooted fans and upper adjacency graphs. The top row shows the first three binary rooted fans, where at each step two new triangular faces are attached to the edges added at the previous step. The bottom row shows the corresponding upper adjacency graphs for the edges. That is, each edge in the fan is represented by a vertex in the graph and two vertices are adjacent if the corresponding edges bound a common face. }
    \label{fig:fans}
\end{figure}

\begin{theorem} Let $F^{k,d}_{n+1}(x)=(x-(d-2))F^{k,d}_n(x)-k(d-1)F^{k,d}_{n-1}(x)$ with $F^{k,d}_0(x)=0$ and $F^{k,d}_1(x)=1$ and $G^{k,d}_n(x)= xF^{k,d}_n(x) - k(d-1)F_{n-1}^{k,d}(x) $.
Then, if $\lambda$ is a root of $G^{k,d}_n$ it is an eigenvalue of the upper adjacency operator of the $k$-ary,  rooted fan of dimension $d$ of depth $r$. 
\end{theorem}

\begin{proof}
 The upper adjacency graphs of these fans can be constructed by beginning with a root node and at every step connecting $k$ copies of the complete graph on $d-1$ nodes to each node that was added at the previous step. These complete graphs are disjoint, except for their common ancestor. 
 
 As in Lemma 2 we construct an eigenvector for $\lambda$  explicitly by labelling each node at level $m$  with $F^{k,d}_{r-(m-1)}(\lambda)$. Then, we verify the eigenrelation in three cases: 
 
 \noindent {\bf Case 1:} $m=r$. In this case we are at the nodes added at the most recent step. These nodes are adjacent to their parent and $d-2$ other newly added nodes. Since $F^{k,d}_{1}(\lambda)=1$ and $F^{k,d}_{2}(\lambda)= (\lambda-(d-2))1-k(d-1)0 = \lambda-{d-2}$ the relation is satisfied. 
\medskip

\noindent {\bf Case 2:} $0 < m < r$. In this case we are at a node that is neither a leaf nor the root and  has $k(d-1)$ children, $(d-2)$ siblings, and one parent. This is exactly the condition imposed by the recurrence relation for $F^{k,d}$.

\medskip
\noindent {\bf Case 3:} $m=0$. In this case, we are at the root which has $k(d-1)$ children and
\begin{align*}
k(d-1) F^{k,d}_r(\lambda) =& \lambda F^{k,d}_r(\lambda)
 \end{align*}
since $\lambda$ is a root of  $G^{k,d}_n$.
 
\end{proof}

Notice that this result reduces to the case of regular trees discussed above with $d=2$ since $F^{k,2}_n = P^k_n$. We leave for future work the problem of computing the densities of these eigenvalues as well as  the graphs and eigenvalues associated to the full simplicial Laplacian.

\subsection{Future Work}
These results motivate several related questions that we will address in future work. One generalization of the construction process of Lemma 4 is the problem of finding subgraphs whose restrictions remain eigenvectors. That is, given a graph $G$ and an associated adjacency eigenpair $(v,\lambda)$ does there exist a subgraph $H$ of $G$ so that $(v\vert_H,\lambda')$ is an eigenpair for $H$? For $\lambda=0$ it is easy to construct examples using leaves or disconnected components and allowing $\lambda \neq \lambda'$ means that regular subgraphs of regular graphs have this property with respect to $v=\mathbf{1}$. Determining other families of examples and characterizing this behavior in terms of properties of $G$, $H$, and $(v,\lambda)$ remains open. 

Another interesting problem relates to the endpoints of the Cantor--like sets that occur in Theorems 4 and 6.  The connection with Lambert series and the nice result of Corollary 2 suggests that a more compact representation might exist, defined in terms of partial sums of the $\varphi$ function.   Finally, one way to abstract the  property that causes the spectral measures of $A(X_n^{k})$ to converge to a singular distribution as $n\rightarrow \infty$ can be characterized as follows. Let $G_1,G_2,\ldots$ be an infinite sequence of graphs with $|G_i|$ increasing. Then, for all $\varepsilon>0$ there must exist a finite set $\Lambda\subset \mathbb{R}$ and a $N\in\mathbb{N}$ such that for all $n>N$:
$$ \dfrac{|\{\lambda\in \operatorname{spec}(G_n): \lambda\notin \Lambda\}| }{|\operatorname{spec}(G_n)|} <\varepsilon.$$  Is it possible to give a combinatorial description of all graph sequences that satisfy this property?
 Note that this property is not sufficient to guarantee that the spectral measures converge. 
 For example, we can consider the sequence interleaving $\xnk$ with complete graphs of size $\dfrac{k^{r+1}}{k-1}$, which fails since all but one of the eigenvalues of $K_n$ are zero for any $n$. 

\section*{Ackowledgements} The first author acknowledges the generous support of the Prof. Amar G. Bose Research Grant.

\end{document}